\documentclass{article}
\usepackage[utf8]{inputenc}
\usepackage{amsmath,amssymb,trimclip,adjustbox}
\usepackage{amsfonts}
\usepackage{amssymb}
\usepackage{amsthm}
\usepackage{amsmath}
\usepackage{mathrsfs}
\usepackage{stmaryrd}
\newtheorem{theorem}{Theorem}[section]

\newtheorem{corollary}[theorem]{Corollary}
\newtheorem{lemma}[theorem]{Lemma}

\newtheorem{mydef}[theorem]{Definition}
\newtheorem{claim}[theorem]{Claim}
\newtheorem{question}[theorem]{Question}

\title{A note on infinite partitions of free products of Boolean algebras}
\author{Mario Jardon Santos}
\date{}

\begin{document}

\maketitle

\begin{abstract}
    If $A$ is an infinite Boolean algebra the cardinal invariant $\mathfrak{a}(A)$ is defined as the smallest size of an infinite partition of $A$. The cardinal $\mathfrak{a}(A\oplus B)$, where $A\oplus B$ is the free product of the Boolean algebras $A$ and $B$ (whose dual topological space is the product of the dual topological spaces of $A$ and $B$), is below both $\mathfrak{a}(A)$ and $\mathfrak{a}(B)$. The equality $\mathfrak{a}(A\oplus B)=\min\lbrace\mathfrak{a}(A),\mathfrak{a}(B)\rbrace$ is not known to hold for all infinite Boolean algebras $A$ and $B$. Here some lower bounds of $\mathfrak{a}(A\oplus B)$ are provided.
\end{abstract}

\section{Introduction}

If $(A,+,\cdot,-,1,0)$, usually abbreviated as $A$, is a Boolean algebra and $P\subseteq A^+ :=A\setminus\lbrace 0\rbrace$, we will say that $P$ is a \textit{disjoint} family of $A$ if $a\cdot b=0$, for all $a,b\in A$. If also for every $x\in A^+$ there exists $a\in P$ such that $a\cdot x\neq 0$, $P$ will be called a \textit{partition} of $A$. If $P\subseteq A^+$ and $\prod_{i\leq k}a_i \neq 0$, for all finite $\lbrace a_i :i\leq k\rbrace\subseteq P$, it will be called a \textit{centered} family of $A$. If $P\subseteq A^+$ is a centered family and there exists $x\in A^+$ such that $x\leq a$, i.e. $x\cdot a=x$, for all $a\in P$, it will be said that $a$ is a \textit{pseudointersection} of $P$. From these concepts the following cardinal invariants are defined for every infinite Boolean algebra $A$: $$\mathfrak{a}(A):=\min\lbrace\vert P\vert:P\subseteq A^+~is~an~infinite~partition\rbrace$$

$$\mathfrak{p}(A):=\min\lbrace\vert P\vert:P\subseteq A^+ ~is~centered~with~no~pseudointersection\rbrace.$$

If $P\subseteq A^+$ is an infinite partition of $A$, it is easy to see that $\lbrace-x:x\in P\rbrace$ is a centered family with no pseudointersection. It follows that $\mathfrak{p}(A)\leq\mathfrak{a}(A)$, for every infinite Boolean algebra $A$.

 If $A$ and $B$ are two Boolean algebras, their free product, denoted $A\oplus B$, is an algebra $
 C$ such that there exist subalgebras $A^\prime ,B^\prime \leq C$, such that $A\cong A^\prime $, $B\cong B^\prime$, $$C=\langle A^\prime \cup B^\prime \rangle:=\left\lbrace\sum_{i<n}a_i \cdot b_i : n<\omega, a_i \in A^\prime , b_i \in B^\prime \right\rbrace$$ and $a\cdot b\neq 0$, for all $a\in A^\prime \setminus\left\lbrace 0\right\rbrace$ and all $b\in B^\prime \setminus\left\lbrace 0\right\rbrace$. Given two Boolean algebras $A$ and $B$, this algebra exists and is unique up to isomorphisms.

In Theorem 11 of \cite{mio} it is proved that $$\mathfrak{p}(A\oplus B):=\min\lbrace\mathfrak{p}(A),\mathfrak{p}(B)\rbrace.$$ Since every partition $P$ of $A$ (resp. $B$) induces a partition of $A\oplus B$, namely $\lbrace a\cdot 1 : a\in P\rbrace$,  it easily follows that $\mathfrak{a}(A\oplus B)\leq\mathfrak{a}(A),\mathfrak{a}(B)$, for all $A$ and $B$ infinite Boolean algebras. In the light of all this, in \cite{monk} (Problem 8) a pretty simple question on this cardinal invariant was asked:

\begin{question}
\label{pregunta_monk}
Does $$\mathfrak{a}(A\oplus B)=\min\lbrace\mathfrak{a}(A),\mathfrak{a}(B)\rbrace$$ hold for any pair of infinite Boolean algebras $A$ and $B$?
\end{question}

A partial answer to this question was given in Theorem 13 of \cite{mio}.

\begin{theorem}
\label{cota1}
If $A$ and $B$ are infinite Boolean algebras, then
 $$\min\left\lbrace \min\left\lbrace\mathfrak{a}\left( A\right) ,\mathfrak{a}\left( B\right)\right\rbrace ,\max\left\lbrace\mathfrak{p}\left( A\right) ,\mathfrak{p}\left( B\right)\right\rbrace\right\rbrace\leq \mathfrak{a}\left(A\oplus B\right).$$
\end{theorem}

Observe that from this theorem it follows that any instance of $$\mathfrak{a}(A\oplus B)<\min\lbrace\mathfrak{a}(A),\mathfrak{a}(B)\rbrace$$ is one of $\mathfrak{p}(A),\mathfrak{p}(B)<\mathfrak{a}(A),\mathfrak{a}(B)$. Since $\mathfrak{a}(A)=\omega$ iff $\mathfrak{p}(A)=\omega$, for every infinite Boolean algebra $A$, for getting such a counterexample we need that $\omega_1 \leq\mathfrak{p}(A),\mathfrak{p}(B)$, and hence that $\omega_1 \leq\mathfrak{a}(A\oplus B)$. In this note a couple of lower bounds to $\mathfrak{a}(A\oplus B)$, for $A$ and $B$ infinite Boolean algebras, will be given, adding some nuance to the bound of Theorem \ref{cota1}. The first one will only work on \textit{homogeneous} Boolean algebras.

\begin{mydef}
A Boolean algebra $(A,+,\cdot,-,1,0)$ will be called homogeneous if for all $x\in A$, the Boolean algebra defined on $A\upharpoonright x:=\lbrace y\in A: y\leq x\rbrace$, with structure $(A\upharpoonright x,+,\cdot,-^\prime ,x,0)$, where $-^\prime y:=(-y)\cdot x$, for all $y\leq x$, is isomorphic to $A$.
\end{mydef}

A famous example of homogeneous Boolean algebra is the quotient $P(\omega)\slash fin$, i.e. the power set of $\omega$ modulo the ideal of finite sets of $\omega$. Now we define a couple of concepts and cardinal invariants that will help with the bounds given in this note.

\begin{mydef}
Let $A$ be an atomless Boolean algebra, i.e. an algebra such that for all $x\in A^+$ there exists $y\in A^+$ such that $y<x$. A family $P\subseteq A^+$ will be called splitting if for all $x\in A^+$ there exists $y\in P$ such that $x\cdot y\neq 0\neq x\cdot(-y)$. We define $\mathfrak{s}(A)$, the splitting number of $A$, as the smallest size of a splitting family of $A$.
\end{mydef}

\begin{mydef}
Let $A$ be an infinite Boolean algebra. A pair $(\mathcal{A},\mathcal{B})$, for $\mathcal{A},\mathcal{B}\subseteq A^+$, will be called a Rothberger gap if $\vert\mathcal{A}\vert=\omega$, $a\cdot b=0$, for all $a\in\mathcal{A}$ and $b\in\mathcal{B}$, and there is no $c\in A^+$ such that $a\cdot c=0$, for all $a\in\mathcal{A}$ and $b\leq c$, for all $b\in\mathcal{B}$. If there exists a Rothberger gap in $A$, then define $$\mathfrak{b}(A):=\min\lbrace\vert\mathcal{B}\vert:\exists\mathcal{A}\subseteq A^+ ~(\mathcal{A},\mathcal{B})~is~a~Rothberger~gap\rbrace.$$
\end{mydef}

Observe that any maximal centered subfamily of a splitting family is a centered family with no pseudointersection. Therefore $\mathfrak{p}(A)\leq\mathfrak{s}(A)$, for all atomless infinite Boolean algebra $A$. Also, if $\mathfrak{a}(A)\geq\omega_1$ and $\lbrace a_\alpha :\alpha<\kappa\rbrace$ is an infinite partition, then $\mathcal{A}:=\lbrace a_n :n<\omega\rbrace$ and $\mathcal{B}:=\lbrace a_\alpha :\alpha\in\kappa\setminus\omega\rbrace$ form a Rothberger gap: otherwise if there is $c\in A^+$ such that $a_n \cdot c=0$, for all $n<\omega$ and $a_\alpha \leq c$, for all $\alpha\in\kappa\setminus\omega$, then $\mathcal{A}\cup\lbrace c\rbrace$ is a partition of $A$, which is a contradiction. Therefore $\mathfrak{b}(A)\leq\mathfrak{a}(A)$, for all infinite Boolean algebra $A$ with no countable partitions.

\section{Lower bounds for $\mathfrak{a}(A\oplus B)$}

Since every Boolean algebra is isomorphic to the algebra of clopen sets of some zero-dimensional compact Hausdorff space, from now on $A$ and $B$ will be respectively the algebra of clopen sets of some zero-dimensional compact Hausdorff spaces $X$ and $Y$. Accordingly $A\oplus B$ will refer to the algebra of clopen sets of the product space $X\times Y$.\footnote{For topological duality, as well as other basic topics on Boolean algebras, the reader is referred to \cite{handbook}.}

Observe that if $c\in A\oplus B$, then there exist $\lbrace a_i : i<k\rbrace\subseteq A$ and $\lbrace b_i : i<k\rbrace\subseteq B$, for $k<\omega$, such that $$c=\bigcup_{i<k}a_i \times b_i .$$ Since the following equality holds:

$$c=\bigcup_{\emptyset\neq J\subseteq k}(\bigcap_{i\in J}a_i \setminus\bigcup_{j\in k\setminus J}a_j )\times\bigcup_{i\in J}b_i$$ we can always assume that either $\lbrace a_i : i<k\rbrace$ is a disjoint family or that $\lbrace b_i : i<k\rbrace$ is a disjoint family. Therefore when dealing with infinite partitions (or disjoint families) of $A\oplus B$ we can always assume that they are of the form $\lbrace a_\alpha \times b_\alpha :\alpha<\kappa\rbrace$, where $a_\alpha \in A$ and $b_\alpha \in B$, for all $\alpha<\kappa$.

\begin{theorem}
\label{cota_s}
Suppose that $A$ and $B$ are homogeneous and that $\omega_1 \leq\mathfrak{a}(A),\mathfrak{a}(B)$. Then $\min\lbrace\mathfrak{a}(A),\mathfrak{a}(B),\max\lbrace\mathfrak{s}(A),\mathfrak{s}(B)\rbrace\rbrace\leq\mathfrak{a}(A\oplus B)$.
\end{theorem}
\begin{proof}
Suppose that $\omega_1 \leq\kappa<\mathfrak{a}(A),\mathfrak{a}(B),\max\lbrace\mathfrak{s}(A),\mathfrak{s}(B)\rbrace$ and that $P=\lbrace a_\alpha \times b_\alpha :\alpha<\kappa\rbrace$ is a disjoint subfamily of  $A\oplus B$. Without loss of generality we can suppose that $\kappa<\mathfrak{s}(A)$. We will prove two cases.

\textit{Case 1.} There exists $E\in[\kappa]^\omega$ such that $\lbrace a_\alpha :\alpha\in E\rbrace$ is a centered family. Without loss of generality $E=\omega$. Since $\omega_1 \leq\mathfrak{p}(A)$, we can take $a^\prime \in A^+$ such that $a^\prime \subseteq a_n$, for all $n<\omega$. Furthermore, take $a\in A\upharpoonright a^\prime$ which witnesses that $\lbrace a_\alpha \cap a^\prime :\alpha<\kappa\rbrace$ is not a splitting family of $A\upharpoonright a^\prime$, i.e. for all $\alpha<\kappa$, either $a\cap a_\alpha =\emptyset$ or $a\subseteq a_\alpha$. Since $E:=\lbrace\alpha<\kappa :a\subseteq a_\alpha \rbrace$ is an infinite set, it follows that $\lbrace b_\alpha :\alpha\in E\rbrace$ is infinite disjoint family of $B$. Also, since $\kappa<\mathfrak{a}(B)$, there exists $b\in B$ such that $b\cap b_\alpha =\emptyset$, for all $\alpha\in E$. Take $\alpha<\kappa$. If $\alpha\in E$, then $b\cap b_\alpha =\emptyset$. If $\alpha\notin E$, then $a\cap a_\alpha =\emptyset$. In either case $a\times b$ is disjoint to $a_\alpha \times b_\alpha$, which means that $P$ is not an infinite partition.

\textit{Case 2.} The family $\lbrace a_\alpha :\alpha\in E\rbrace$ is not centered, for all $E\in[\kappa]^\omega$. The family $\lbrace a_\alpha :\alpha<\kappa\rbrace$ is not splitting. Note that if $c\in A^+$ witnesses this fact, so does every $0\neq c^\prime \subseteq c$. Since $A$ is homogeneous, if $\lambda=\vert A\vert$, there exists $\lbrace c_\gamma :\gamma<\lambda\rbrace\subseteq A^+$ such that for all $\alpha<\kappa$ and all $\gamma<\lambda$ either $c_\gamma \subseteq a_\alpha$ or $c_\gamma \cap a_\alpha =\emptyset$. For all $\gamma<\lambda$ the set $\lbrace\alpha<\kappa:c_\gamma \subseteq a_\alpha \rbrace$ is finite, by hypothesis. Then we can define $f:\lambda\rightarrow[\kappa]^{<\omega}$ such that $c_\gamma \subseteq a_\alpha$ iff $\alpha\in f(\gamma)$, for all $\alpha<\kappa$ and $\gamma<\lambda$.

\begin{claim}
One of the following statements holds:

\begin{itemize}
    \item $\emptyset\in f[\lambda]$

    \item  there exists $E\in f[\lambda]$ such that $b:=Y\setminus\bigcup_{\alpha\in E}b_\alpha$ is not empty

    \item there exist $E\in f[\lambda]$ and $\beta\in\kappa\setminus E$ such that $\lbrace a_\alpha :\alpha\in E\cup\lbrace\beta\rbrace\rbrace$ is a centered family.
\end{itemize}
\end{claim}

\begin{proof}
Suppose that $\emptyset\notin f[\lambda]$, that $\lbrace a_\alpha :\alpha\in E\cup\lbrace\beta\rbrace\rbrace$ is not centered, for all $E\in f[\lambda]$ and all $\beta\in\kappa\setminus E$, and that $Y=\bigcup_{\alpha\in E}b_\alpha$, for all $E\in f[\lambda]$. For $E\in f[\lambda]$, define $d_E :=\bigcap_{\alpha\in E}a_\alpha$.  Observe that $\lbrace d_E :E\in f[\lambda]\rbrace$ is a disjoint family. If $f[\lambda]$ is finite and $X=\bigcup_{E\in f[\lambda]}d_E$, this means that $\lbrace a_\alpha \times b_\alpha :\alpha\in\bigcup f[\lambda]\rbrace$ covers all $X\times Y$, which is a contradiction. Since $\kappa<\mathfrak{a}(A)$, this means that either if $f[\lambda]$ is finite or not, there exists $c\in A^+$ such that $c\cap d_E =\emptyset$, for all $E\in f[\lambda]$.

Since $A$ is homogeneous and $\kappa<\mathfrak{s}(A)$, then there exists $\gamma<\lambda$ such that $c_\gamma \subseteq c$. If $E=f(\gamma)$, then $c_\gamma \subseteq d_E$, but this is a contradiction.

\end{proof}

If there exists $\gamma<\lambda$ such that $f(\gamma)=\emptyset$, then $c_\gamma \times Y$ witnesses that $P$ is not a partition. Suppose that this is not the case and for each $E\in f[\lambda]$ choose $c_E =c_\gamma$, for some $\gamma\in f^{-1}[E]$. If there exists $E\in f[\lambda]$ such that $b:=Y\setminus\bigcup_{\alpha\in E}b_\alpha$ is not empty, then $c_E \times b$ witnesses that $P$ is not a partition. If there exist $E\in f[\lambda]$ and $\beta\in\kappa\setminus E$ such that $\lbrace a_\alpha :\alpha\in E\cup\lbrace\beta\rbrace\rbrace$ is a centered family, then $c_E \times b_\beta$ witnesses that $P$ is not a partition.

%
\end{proof}

Since $\max\lbrace\mathfrak{p}(A),\mathfrak{p}(B)\rbrace\leq\max\lbrace\mathfrak{s}(A),\mathfrak{s}(B)\rbrace$, for all homogeneous Boolean algebras $A$ and $B$, this more specific theorem gives an improvement to Theorem \ref{cota1}.

\begin{lemma}
\label{caso_infinito}
Suppose 
that $P=\lbrace a_\alpha \times b_\alpha :\alpha <\kappa\rbrace$ is an infinite partition of $A\oplus B$. Then there exists $\lbrace\alpha_n : n<\omega\rbrace\subseteq\kappa$ such that either $\lbrace a_{\alpha_n} : n<\omega\rbrace$ is a centered family or $\lbrace b_{\alpha_n} : n<\omega\rbrace$ is a centered family.
\end{lemma}

\begin{proof}
Suppose that if $E\subseteq\kappa$ is such that $\lbrace a_\alpha :\alpha\in E\rbrace$ is a centered family, then $\vert E\vert <\omega$. 
Observe that if $E$ is maximal with this property, then $\lbrace b_\alpha :\alpha\in E\rbrace$ is a disjoint family and $$\bigcup_{\alpha\in E} b_\alpha =Y.$$  Otherwise, if $a:=\bigcap_{\alpha\in E}a_\alpha$ and $b:=Y\setminus\bigcup_{\alpha\in E}b_\alpha$, then $a\times b$ would witness that $P$ is not a partition.

We will recursively construct a sequence $\lbrace\alpha_n :n<\omega\rbrace$, such that $\lbrace b_{\alpha_n}:n<\omega\rbrace$ is a centered family. Extend $\lbrace 0\rbrace$ to a set $E_0$, maximal with the property that $\lbrace a_{\alpha}:\alpha\in E_0 \rbrace$ is centered.  Enumerate $E_0 =\lbrace\alpha_{0}^0 ,...,\alpha_{0}^{k_0 -1}\rbrace$ and define $H_{0}^i :=\lbrace\beta\in\kappa\setminus E_0 :b_{\alpha_{0}^i}\cap b_\beta \neq\emptyset\rbrace,$ for all $i<n_0$.  There exists $i_0 <k_0$ such that $\vert H_{0}^{i_0} \vert=\kappa$. Define $\alpha_0 :=\alpha_{0}^{i_0}$ and $H_0 :=H_{0}^{i_0}$.

Suppose now that for some $n\geq 1$ we have constructed $\lbrace\alpha_l :l<n\rbrace\subseteq\kappa$, $\lbrace H_l :l<n\rbrace\subseteq[\kappa]^\kappa$, and $\lbrace E_l :l<n\rbrace\subseteq[\kappa]^{<\omega}$ such that \begin{itemize}
    \item $\alpha_l \in E_l$,
    \item $\alpha_{l^\prime}\neq\alpha_{l}$ and
    \item $H_l \subseteq H_{l^\prime}$, for all $l^\prime <l<n$, and that
    \item $$b_\beta \cap\bigcap_{l<k}b_{\alpha_l}\neq\emptyset,$$ iff  $\beta\in H_{n-1}$, for all $\beta<\kappa$.

\end{itemize} Define $b:=\bigcap_{l<k}b_{\alpha_l}$ and take $\beta\in H_{n-1}\setminus\lbrace\alpha_l :l<n\rbrace$. Extend $\lbrace\beta\rbrace$ to a family $E_n$, maximal with the property that $\lbrace a_\alpha :\alpha\in E_n \rbrace$ is centered. Clearly $$b=\bigcup_{\alpha\in E_n \cap H_{n-1}}b\cap b_\alpha .$$ Also $\alpha\neq\alpha_{l}$, for all $l<n$ and all $\alpha\in E_n$: otherwise we would have $b_\alpha \cap b_\beta \neq\emptyset\neq a_\alpha \cap a_\beta$, which is a contradiction.  Enumerate $E_n \cap H_{n-1}:=\lbrace\alpha_{n}^0 ,...,\alpha_{n}^{k_n -1}\rbrace$ and define $$H_{k}^i :=\lbrace\beta\in H_{k-1}:b_{\alpha_{n}^i}\cap b\cap b_\beta \neq\emptyset\rbrace.$$ Since there exists $i_n <k_n$ such that $\vert H_{n}^{i_n}\vert=\kappa$, define $\alpha_n :=\alpha_{n}^{i_n}$ and $H_{n}:=H_{n}^{i_n}$. Since we can continue this recursion, we get $\lbrace\alpha_n :n<\omega\rbrace\subseteq\kappa$ such that $\lbrace b_{\alpha_n}:n<\omega\rbrace$ is a centered family.
\end{proof}

\begin{theorem}
\label{cota_b}
Suppose that 
$\omega_1 \leq\mathfrak{a}(A),\mathfrak{a}(B)$. Then  $\min\lbrace\mathfrak{b}(A),\mathfrak{b}(B)\rbrace\leq\mathfrak{a}(A\oplus B)$.
\end{theorem}

\begin{proof}
Suppose that $\kappa<\mathfrak{b}(A),\mathfrak{b}(B)$ and that  $P=\lbrace a_\alpha \times b_\alpha  :\alpha<\kappa\rbrace$ is an infinite partition of $A\oplus B$. From Lemma \ref{caso_infinito}, without loss of generality we can suppose that $\lbrace a_n :n<\omega\rbrace$ is a centered family and that $\lbrace b_n :n<\omega\rbrace$ is a pairwise disjoint family. Take $a\in A^+$ such that $a\subseteq a_n$, for all $n<\omega$.  Define $E:=\lbrace\alpha\in\kappa\setminus\omega: a_\alpha \cap a\neq\emptyset\rbrace$. Therefore $b_n \cap b_\alpha =\emptyset$ for all $n<\omega$ and $\alpha\in E$. Since $\vert E\vert<\mathfrak{b}(B)$, there exist $c\in B$ such that $b_\alpha \subseteq c$, for all $\alpha \in E$, and $b_n \cap c=\emptyset$, for all $n<\omega$. Since $\lbrace b_n :n<\omega\rbrace\cup\lbrace c\rbrace$ is not an infinite partition of $B$, take $b\in B^+$ as a witness of this fact. Take $\alpha<\kappa$. If $\alpha\in\omega\cup E$, then $b\cap b_\alpha =\emptyset$. If $\alpha\notin\omega\cup E$, then $a\cap a_\alpha =\emptyset$. Either way $a\times b$ witnesses that $P$ is not an infinite partition.

\end{proof}



This result is not precisely an improvement of Theorem \ref{cota1} on a broad class of infinite Boolean algebras. Nevertheless, some of its applications definitely are. Now these theorems will be applied to the more familiar case when $A=B=P(\omega)\slash fin$. So a word on its cardinal invariants will be given.

Recall that $\mathfrak{a}$ is the least size of an infinite \textit{maximal almost disjoint} (mad) family, i.e. a family $\lbrace A_\alpha :\alpha<\kappa\rbrace\subseteq[\omega]^\omega$, such that $\vert A_\alpha \cap A_\beta \vert<\omega$, for all $\alpha<\beta<\kappa$, and that for all $X\in[\omega]^\omega$ there exists $\alpha<\kappa$ such that $\vert X\cap A_\alpha \vert=\omega$. Also the number $\mathfrak{s}$ is defined as the least size of a \textit{splitting} family. i.e. a family $\lbrace A_\alpha :\alpha<\kappa\rbrace\subseteq[\omega]^\omega$, such that for all $X\in[\omega]^\omega$ there exists $\alpha<\kappa$ such that $\vert X\cap A_\alpha \vert=\vert X\setminus A_\alpha \vert=\omega$. It is an easy observation that $\mathfrak{a}(P(\omega)\slash fin)=\mathfrak{a}$ and that $\mathfrak{s}(P(\omega)\slash fin)=\mathfrak{s}$.

Similarly the cardinal $\mathfrak{b}$ can be defined to be $\mathfrak{b}(P(\omega)\slash fin)$, though its usual definition is as the smallest size of a family $\mathcal{F}\subseteq\omega^\omega$ such that for all $g\in\omega^\omega$ there exists $f\in\mathcal{F}$ such that $f(m)> g(m)$, for infinitely many $m<\omega$. The equivalence of both definitions was proved in \cite{rothberger}.

The cardinal $\mathfrak{a}(1):=\mathfrak{a}$ is the first (but not least) element of related cardinal invariants.
Observe that when dealing with the infinite partitions of the finite free products $\bigoplus_{i<n}P(\omega)\slash fin$, for $2\leq n<\omega$, we can always assume that they are of the type $$\lbrace\prod_{i<n}X^{\alpha}_i :\alpha<\kappa\rbrace,$$ where each $X^{\alpha}_i$ is a non-empty clopen set of $\beta\omega\setminus\omega$. Since $$\prod_{i<n}X^{\alpha}_i \cap\prod_{i<n}X^{\beta}_i =\emptyset$$ iff there exists $i<n$ such that $X^{\alpha}_i \cap X^{\beta}_i$, and each $X^{\alpha}_i$ can also be thought of as an infinite set of $\omega$, the following definition gives us an infinite combinatorics way to approach these infinite partitions.

\begin{mydef}
\label{n-mad}

Take $2\leq n<\omega$. An infinite family $\lbrace (X^{0}_\alpha ,...,X^{n-1}_\alpha ): \alpha <\kappa\rbrace\subseteq([\omega]^\omega )^n$ is called an $n$-ad family if for all $\alpha <\beta <\kappa$ there exists $i<n$ such that $\vert X^{i}_\alpha \cap X^{i}_\beta \vert <\omega$. It will be called an $n$-mad family if it is maximal with this property. Define $\mathfrak{a}(n)$ as the smallest size of an $n$-mad family.
\end{mydef}

\begin{corollary}
The following statements hold:
\begin{enumerate}

\item $\omega_1 \leq\mathfrak{a}(n+1)\leq\mathfrak{a}(n)\leq\mathfrak{a}$, for all $1\leq n<\omega$.

\item $\min\lbrace\mathfrak{s}, \mathfrak{a}(n-1)\rbrace\leq\mathfrak{a}(n)$, for all $2\leq n<\omega$.

\item $\mathfrak{b}\leq\mathfrak{a}(2)$.
\end{enumerate}
\end{corollary}

\begin{proof}
\textit{1} follows from definition and $\omega_1 \leq\mathfrak{a},\mathfrak{s}$. Since $\mathfrak{s}(A\oplus B)=\min\lbrace\mathfrak{s}(A),\mathfrak{s}(B)\rbrace$, for all infinite atomless Boolean algebras $A$ and $B$ (see \cite{mio}), \textit{2} follows from Theorem \ref{cota_s} and \textit{1}. \textit{3} follows from Theorem \ref{cota_b}.
\end{proof}

Going back to Question \ref{pregunta_monk} and focusing on the case of the finite free products of $P(\omega)\slash fin$, this corollary implies that if any of these products is (consistently) a counterexample of said equality, it has to be in a model of $\mathfrak{s}<\mathfrak{a}$. In the specific case of $P(\omega)\slash fin\oplus P(\omega)\slash fin$, also $\mathfrak{b}<\mathfrak{a}$ must hold in the model. Besides these observations, the existence of these models remain open.

\begin{question}
\begin{enumerate}
    \item  Is it consistent that $\omega_1 =\mathfrak{s}=\mathfrak{a}(n)<\mathfrak{a(n-1)}=\omega_2$, for any $2\leq n<\omega$?

    \item Is it consistent that $\omega_1 =\mathfrak{s}=\mathfrak{b}=\mathfrak{a}(2)<\mathfrak{a}=\omega_2 ?$
\end{enumerate}
\end{question}

\end{document}